\documentclass[11pt]{amsart}
\usepackage{amsmath,amssymb,latexsym,esint,cite,mathrsfs}
\usepackage{verbatim,wasysym}
\usepackage[left=2.7cm,right=2.7cm,top=2.9cm,bottom=2.9cm]{geometry}

\usepackage{microtype}
\usepackage{color,enumitem,graphicx}
\usepackage[colorlinks=true,urlcolor=blue, citecolor=red,linkcolor=blue,
linktocpage,pdfpagelabels, bookmarksnumbered,bookmarksopen]{hyperref}
\usepackage[hyperpageref]{backref}
\usepackage[english]{babel}

\newtheorem{lemma}{Lemma}[section]

\newtheorem{proposition}[lemma]{Proposition}
\newtheorem{theorem}[lemma]{Theorem}

\theoremstyle{definition}

\newtheorem{remark}[lemma]{Remark}
\newtheorem{example}[lemma]{Example}

\numberwithin{equation}{section}

\usepackage{orcidlink} 
\usepackage[square, comma, numbers, sort&compress]{natbib} 

\newcommand{\R}{\mathbb{R}}
\newcommand{\mc}[1]{\mathcal{#1}}
\newcommand{\dive}{{\rm div}}

\newcommand{\sgn}{{\rm sgn}}
\newcommand{\nequiv}{\not \equiv}
\def\eps{\varepsilon}

\def\pabs#1{\left|{#1}\right|}
\def\norm#1{\|#1\|}

\renewenvironment{proof}[1][\proofname]{\medskip \noindent {\bfseries #1. }}{\qed \bigskip}

\title[Concavity for quasilinear equations and optimality remarks]{Concavity properties for quasilinear 
equations \\ and optimality remarks}

\author[N.\ M.\ Almousa]{Nouf M. Almousa \orcidlink{0000-0001-6412-1993}}
\author[J.\ Assettini]{Jacopo Assettini}
\author[M.\ Gallo]{Marco Gallo \orcidlink{0000-0002-3141-9598}}
\author[M.\ Squassina]{Marco Squassina \orcidlink{0000-0003-0858-4648}}

\address[N.\ M.\ Almousa, M.\ Squassina]{\newline\indent College of Science
	\newline\indent
	Princess Nourah Bint Abdul Rahman University
	\newline\indent
	Saudi Arabia, Riyadh, PO Box 84428}
\email{\href{mailto:nmalmousa@pnu.edu.sa}{nmalmousa@pnu.edu.sa}}
\email{\href{mailto:marsquassina@pnu.edu.sa}{marsquassina@pnu.edu.sa}}

\address[J.\ Assettini, M.\ Gallo, M.\ Squassina]{\newline\indent Dipartimento di Matematica e Fisica
	\newline\indent
	Università Cattolica del Sacro Cuore
	\newline\indent
	Italy, Brescia, BS, Via della Garzetta 48, 25133}
\email{\href{mailto:jacopo.assettini01@icatt.it}{jacopo.assettini01@icatt.it}} 
\email{\href{mailto:marco.gallo1@unicatt.it}{marco.gallo1@unicatt.it}}
\email{\href{mailto:marco.squassina@unicatt.it}{marco.squassina@unicatt.it}}

\thanks{The first and fourth authors are supported by Princess Nourah bint Abdulrahman University Researchers Supporting Project number (PNURSP-HC2023/3), Princess Nourah bint Abdulrahman University, Saudi Arabia.
The third and fourth authors are members of {Gruppo Nazionale per l'Analisi Ma\-te\-ma\-ti\-ca, la Probabilit\`a e le loro Applicazioni} (GNAMPA). 
	The third author is supported by PRIN 2017JPCAPN ``Qualitative and quantitative aspects of nonlinear PDEs''.
		}

\subjclass[2010]{%
Concavity of solutions, 
Quasi-linear elliptic equations.
}
\begin{document}

\begin{abstract}
In this paper we study quasiconcavity properties of solutions of Dirichlet problems related to modified nonlinear Schrödinger equations of the type
$$-{\rm div}\big(a(u) \nabla u\big) + \frac{a'(u)}{2} |\nabla u|^2 = f(u) \quad \hbox{in $\Omega$},$$
where $\Omega$ is a convex bounded domain of $\mathbb{R}^N$.
In particular, we search for a function $\varphi:\mathbb{R} \to \mathbb{R}$, modeled on $f\in C^1$ and $a\in C^1$, which makes $\varphi(u)$ concave. Moreover, we discuss the optimality of the conditions assumed on the source.
\end{abstract}

\maketitle

\begin{center}
	\begin{minipage}{10cm}
		\small
		\tableofcontents
	\end{minipage}
\end{center}

\medskip

\section{Introduction}



A natural question in studying differential equations is whether the solutions of a PDE, set in a certain domain, inherit or not the geometrical properties of the domain itself. 
Starting from \cite{GNN79} 
a huge amount of work has been developed in achieving symmetry of solutions 
from the symmetry (and convexity) of the domain, especially exploiting the Alexandroff-Serrin moving plane method. If the domain is a ball, for example, the solution 
reveals to be radially symmetric and decreasing: in this case all the level sets of the solution are balls and, thus, convex sets.

When the symmetry of the domain is dropped, one may wonder if the solutions still inherit convexity (or star-shapedness) of the domain: 
this question has been addressed starting from the pioneering papers \cite{PS51, 
ML71, BL76, Lew77}, 
and further developed in the subsequent years. 
We highlight that, contrary to star-shapedness, weaker properties of the domain (like simple connection) are generally not inherited by the level sets of the solutions (see e.g. \cite{Sto64}).

Convexity properties of solutions are actually a good information which arise, for example, 
in the study of the free boundaries and the coincidence sets in obstacle problems \cite{Kin78, CS82}, 
in minimal surfaces and prescribed Gauss curvature problems \cite{GuMa05}, 
as well as in optimal control (ensuring, for example, that the trajectories are contained in the domain);
they further give information on the gap between eigenvalues \cite{SWYY},
and on the critical points and the uniqueness of solutions. 
These properties additionally appear in several applications, such as 
plasma confinement \cite{Ack81}, 
capillary surfaces \cite{Kor83Ca}, 
capacitory potentials \cite{Lew77}, 
porous solid combustion \cite{KS87}, 
economy \cite{SZ81}, 
and elasto-plastic deformation of cylinders \cite{ML71}. 

When searching for concavity properties of solutions, one can easily observe that, generally, concavity itself is a too strong goal: while for the torsion problem the concavity may be obtained, for example, for suitable perturbations of ellipsoids \cite{HNST18} (see also \cite{Kos87}), 
in \cite[Remark 2.7]{Kor83Co} 
(see also \cite{Lind94})
it is shown that the first eigenfunction 
of the Laplacian is 
never concave, whatever the bounded regular set is (see \cite[Remark 3.4]{Kaw85R} and Remark \ref{rem_never_concav} for a more general framework).
The same holds true anyway also for the torsion problem, if for example the boundary has some flat zone \cite[Theorem 18]{Ken84Th}. 


Generally, one may search instead for a strictly increasing function $\varphi$ such that the composition $\varphi(u)$ with the solution $u$ is actually concave. This property is generally stronger than the \emph{quasiconcavity}, which requires that all the level sets of $u$ are concave; both the properties have been extensively investigated in literature.


\smallskip

In the present paper we study concavity properties of solutions to the following quasilinear Dirichlet boundary problem 
\begin{equation}
	\label{eq_main_intr}
\begin{cases}
-\dive\big(a(u) \nabla u\big) + \frac{a'(u)}{2} |\nabla u|^2 = f(u) & \Omega,
\\
u>0 & \Omega,
\\
u=0 & \partial \Omega
\end{cases}
\end{equation}
related to the so called \emph{modified nonlinear Schrödinger equation} (MNLS). In particular, we investigate how the weight $a \in C^1((0,+\infty))$ and the source $f\in C^1((0,+\infty))$ influence the ``concaving'' function $\varphi:\R \to \R$. 

\smallskip

Focusing on the semilinear case $a\equiv 1$, in the seminal paper \cite{ML71} 
it is shown that the solutions of the torsion problem (i.e. 
$f \equiv 1$) are such that $\varphi(u)=\sqrt{u}$ is concave; moreover, the authors in \cite{BL76} 
show that the eigenfunctions of the first eigenvalue of Laplacian (i.e. $f(t) = \lambda_1 t$) satisfy $\varphi(u)=\log(u)$ concave. 
More generally, it has been shown \cite{Ken85} that solutions of $-\Delta u = u^q$, $q \in (0,1)$, verify $\varphi(u)=u^{\frac{1-q}{2}}$ concave: 
in all the three cases we see that (up to multiplicative and additive constants)
$$\varphi(t) \equiv \int_1^t \frac{1}{\sqrt{F(s)}} ds$$
where $F(s):= \int_0^s f(\tau) d \tau$ is the antiderivative of the source $f$.

Several papers have then further investigated concavity properties of functions, both on convex domains and convex annuli, both for semilinear and quasilinear equations \cite{Lio81, 
CS82, Kor83Co, Kaw85W, Ken85, 
Sak87, 
Lin94, 
Gom07} (see also \cite{Kaw87
} for some results on complements of bounded sets): 
the proofs involve different techniques, 
such as maximum principles for suitable concavity functions, parabolic and probabilistic methods, convex rearrangements.
We refer to \cite{Ken84Th, Kaw85R, Kaw86G, KL87}
for some surveys on the topic up to the `80s.

We further observe that these results have been generalized to 
parabolic frameworks (in the sense that the quasiconcavity of the initial datum is conserved in time, see 
\cite{DK93} 
and references therein) 
singularly perturbed equations \cite{GrMo03}, 
Hessian equations 
(see \cite{Ye13} 
and references therein),
manifolds \cite{Shi89},
and many other frameworks. 
See also \cite{Ack81, Tal81, LS90Ip,LV08} 
for results about existence of a single (quasi)concave solution (mainly by variational methods through constraints which naturally include quasiconcavity), 
and \cite{BS20} 
for results about concavity up to an error.

Most of the cited papers, anyway, study quasiconcavity of solutions, or give assumptions on the source $f$ in order to have a suitable power $u^{\gamma}$, $\gamma \in (0,1]$, of the solution $u$, to be concave (or $\log$-concave, formally $\gamma=0$). Recently, Borrelli, Mosconi and the fourth author \cite[Theorem 1.2]{BMS22} proved instead the following result, which shows concavity for a suitable $\varphi$ \emph{directly connected} to the source $f$ and to the operator $-\Delta_p$ involved (see also \cite{CF85, Kaw85W}). 
We refer to \cite[Remark 1.6]{BMS22} for some technical comments. 

\begin{theorem}[\cite{BMS22}] \label{thm_BMS}
Let $N\geq 1$, $\Omega \subset \R^N$ be a bounded, convex domain with $C^2$-boundary, and let $f \in C^{\sigma}_{{\rm loc}}([0,+\infty))$ for some $\sigma \in (0,1]$. 
Let $p>1$ and $u\in W^{1,p}_0(\Omega)$ be a weak solution of
$$\begin{cases}
-\Delta_p u = f(u) & \Omega,
\\
u>0 & \Omega,
\\
u=0 & \partial \Omega.
\end{cases}$$
Set
$$M_f:=\inf \big\{ t>0 \mid f(t) = 
0\big\} >0$$
and
$$\varphi(t):= \int_1^t \frac{1}{F^{1/p}(s)} ds, \quad t \in (0, M_f)$$
where $F(t):= \int_0^t f(\tau) d\tau$ for $t \in [0,+\infty)$. 
Assume $f >0$ on $(0,M_f)$, 
and in addition
\begin{itemize}
\item[i)] $F^{1/p}$ concave on $(0,M_f)$,
\item[ii)] $\frac{F}{f}$ convex on $(0,M_f)$.
\end{itemize}
Then $\varphi(u)$ is concave. In particular, $u$ is quasiconcave.
\end{theorem}

%


A first goal is the extension of Theorem \ref{thm_BMS} to the quasilinear case \eqref{eq_main_intr}.
Equation in \eqref{eq_main_intr}, which can be rewritten also as
\begin{equation*}\label{eq_int_a_u}
-a(u) \Delta u - \frac{a'(u)}{2} |\nabla u|^2 = f(u), 
\end{equation*}
naturally arises by considering the (formal) Euler-Lagrange equation of the energy functional
$$u \mapsto \frac{1}{2} \int_{\Omega} a(u) |\nabla u|^2 - \int_{\Omega} F(u)$$
where $F(t)= \int_0^t f(\tau) d \tau$. 
By writing $a(t) \equiv 1 \pm 2 t^2 (\ell'(t^2))^2$, with $\ell \in C^2(\R)$, the equation takes the form
\begin{equation}\label{eq_intr_h}
-\Delta u \mp u \, \ell'(u^2) \Delta(\ell(u^2))= f(u); 
\end{equation}
apart from the classical semilinear case $\ell \equiv const$ or $\ell (t) = \sqrt{t}$ (i.e. $a\equiv const$), several physical models have been developed through equations of the type \eqref{eq_intr_h}.
For example, when $\ell(t)=t$ (i.e. $a(t)=1 + 2t^2 \geq 1$), the classical \emph{MNLS equation}
\begin{equation}\label{eq_intr_MNLS}
-\Delta u - u \Delta(u^2)=f(u)
\end{equation}
arises 
in the fluid theory of upper-hybrid solitons formation \cite{PG76}, 
in the study of electron-phonon interactions \cite{BEPZ01}, 
and in the excitation on a hexagonal lattice to describe fullerenes and nanotubes \cite{BH06} 
($f(s)=s^3$),
as well as 
 in the dynamics of condensate wave functions in superfluid films (\cite{Kur81}, 
$f(s)=s+\frac{s}{(1+s^2)^3}$). 
See also 
\cite{LS78} 
where collapse of quasimonochromatic oscillations and plasma waves is studied. 

By choosing instead $\ell(t)=\sqrt{1+t}$ (i.e. $a(t)=1 - \frac{t^2}{2(1+t^2)} \geq \frac{1}{2}$) and $f(t) = t- \frac{t}{\sqrt{1+t^2}}$ we get the so called \emph{relativistic nonlinear Schrödinger equation}
\begin{equation}\label{eq_intr_relSch}
-\Delta u - \left( 1- \frac{2+\Delta\sqrt{1+u^2}}{\sqrt{1+u^2}}\right)u=0
\end{equation}
which appears related to the self-focusing and channel formation in the nonparaxial propagation of short intense laser pulse through an underdense plasma (governed by charge-displacement due to the ponderomotive force) \cite{BG93, CS93}.
A different model is instead given by $\ell(t)= \sqrt{1-t}$ (i.e. $a(t)= 1 + \frac{t^2}{2(1-t^2)} \geq 1$) and $f(t) = t- \frac{t}{\sqrt{1-t^2}}$, 
\begin{equation}\label{eq_intr_Heisenb}
-\Delta u + u\frac{\Delta\sqrt{1-u^2}}{2\sqrt{1-u^2}} = u - \frac{u}{\sqrt{1-u^2}}
\end{equation}
which arises in the study of excitations in classical planar Heisenberg ferromagnets \cite{TH81}. 

\smallskip

Mathematically, equation \eqref{eq_intr_h} has been extensively studied on the entire space $\Omega=\R^N$ \cite{Col03, CJ04}; 
we refer also to \cite{PSW02, 
Moa06, DMS07, RS10} 
for some results involving external potentials, 
and to \cite{
LPT99} 
for some classical results about dynamical properties.
For results on bounded domains \eqref{eq_main_intr} we refer instead to \cite{LZ13, LLW13}
(see also \cite{GS13} 
for explosive solutions).

\smallskip
As regards concavity properties for \eqref{eq_main_intr}, by reading the equation as 
\begin{equation*}
	-\Delta u = \tfrac{a'(u)}{2 a(u)} |\nabla u|^2 + \tfrac{f(u)}{a(u)} =: B(u,\nabla u)
\end{equation*}
in \cite{BS13} 
it has been showed that, if $\Omega$ is a convex domain and $u$ is a solution of \eqref{eq_main_intr} with $\partial_{\nu}u >0$ on the boundary, then $u$ is $\gamma$-concave, for some $\gamma\leq 1$, provided that the function
\begin{equation}\label{eq_int_old_result}
t \in [0, +\infty) \mapsto \frac{a'(t^{\frac{1}{\gamma}})}{2a(t^{\frac{1}{\gamma}})} t^{1+\frac{1}{\gamma}} \beta + \frac{f(t^{\frac{1}{\gamma}})}{a(t^{\frac{1}{\gamma}})} t^{3-\frac{1}{\gamma}}
\end{equation}
is concave for every $\beta \geq 0$ (see also \cite{Kor83Ca, Kor83Co, Ken85}); see Remark \ref{ex_fisico_power} for some comments. 

\smallskip

In the present paper, we obtain a concavity result for solutions of problem \eqref{eq_main_intr}, where the role of $f$ and $a$ (i.e. of the source and the operator) naturally arises not only in the assumptions, but also in the concaving function $\varphi$. We thus prove the following (see also Theorem \ref{thm_main_concave} and Remark \ref{rem_general_explod} for some generalization).

\begin{theorem} \label{thm_corol_main_concave}
Let $N\geq 1$, $\Omega \subset \R^N$ be a bounded, convex domain with $C^2$-boundary, $a\in C^1([0,+\infty))$ satisfying for some $\nu>0$
\begin{equation}\label{eq_intr_elliptic}
a(t) \geq \nu >0, \quad \hbox{ for any $t \in \R$},
\end{equation}
and $f \in C^{\sigma}_{\rm loc}([0,+\infty))\cap C^1((0,+\infty))$ for some $\sigma \in (0,1]$. 
Let $u \in C(\overline{\Omega}) \cap C^2(\Omega)$ 
be a classical solution of \eqref{eq_main_intr}. Set
\begin{align*}
\xi(t) &:= \frac{F(t)}{f(t)} \frac{a'(t)}{a(t)}, \quad \quad t \in (0, M_f), \\
&= \frac{a'(t)}{F'(t)} \left(\frac{a(t)}{F(t)}\right)^{-1},
\end{align*}
where $F(t):= \int_0^t f(\tau) d\tau$ and
$$M_f:=\inf \{ t>0 \mid f(t) = 
0\}>0.$$
Assume moreover $f >0$ on $(0,M_f)$ and 
\begin{itemize}
\item[i)] $\sqrt{F}$ is concave on $(0, M_f)$;
\item[ii)] $\frac{F}{f}$ is convex on $(0, M_f$);
\item[iii)] $\xi$ is non-decreasing on $(0, M_f)$ with $\lim_{t \to 0^+} \xi(t) \leq 0$.
\end{itemize}
Then $\varphi(u)$ is concave, where 
$$ \varphi(t):= \int_{\mu}^t \sqrt{\frac{a(s)}{F(s)}} ds, \quad t \in (0,+\infty),$$
$\mu:= \nu^{-\frac{1}{2}}$. In particular, $u$ is quasiconcave.
%
\end{theorem}

As proved in \cite[Theorem 2.2.5]{CD94}, when the growth of $f$ at infinity is at most critical, 
each weak solution of problem \eqref{eq_main_intr} 
belongs to $L^\infty(\Omega)$.
Then some nice smoothness results ($C^2$ in our case) when $a$, $f$ and $\Omega$ are smooth follow from \cite{LU68}. This justifies the request on the solution $u$ to be classical in Theorem \ref{thm_corol_main_concave}.

Clearly, when $a \equiv 1$ we recover Theorem \ref{thm_BMS} with $p=2$ and $f\in C^1$, being $\xi \equiv 0$. More generally, we see that, even if conditions $i)$ and $ii)$ of Theorem \ref{thm_corol_main_concave} match with conditions $i)$ and $ii)$ of Theorem \ref{thm_BMS}, nevertheless the function $a$ influences both the assumptions (by requiring in addition condition $iii)$) and the concaving function $\varphi$.

We remark that, apart from the naturalness of the function $\varphi$, to the author's knowledge this function is actually the first tentative of finding a $\varphi$ which makes solutions of \eqref{eq_main_intr} concave: as a matter of fact, it is not known if powers of the solutions are concave or not, when $a \nequiv const$. See Remark \ref{ex_fisico_power} for further comments.

We notice that, by the ellipticity condition \eqref{eq_intr_elliptic}, $a$ plays no relevant role for the asymptotic behavior of $\varphi$ in the origin $t\to 0$; namely, by de l'Hôpital theorem we have
$$\lim_{t \to 0} \frac{ \int_{\mu}^{t} \sqrt{\frac{a(s)}{F(s)}} ds}{ \int_1^t \frac{1}{\sqrt{F(s)}} ds} = \lim_{t \to 0} \frac{ \sqrt{\frac{a(t)}{F(t)}} }{ \frac{1}{\sqrt{F(t)}} } = \sqrt{a(0)} 
 \in (0,+\infty).$$
On the other hand, the behaviour of $\varphi$ at infinity is of no importance, since it is applied to bounded solutions $u \in L^{\infty}(\R)$: thus, the role played by $a$ is related only to the pointwise shape of the function $\varphi$, and it is felt more in the case of large solutions (see Figure \ref{fig_graf_a}).

The idea of the proof rely on a suitable change of variable \cite{CJ04,GS13} 
which allows to bring the problem back to a semilinear one.


\medskip

A second goal of this paper is to show that the assumptions $i)$ and $ii)$ in Theorems \ref{thm_BMS} and \ref{thm_corol_main_concave} are not merely technical, by indicating an example with $a\equiv 1$ where $i$-$ii)$ do not hold, and a non-quasiconcave solution indeed exists; we use some ideas contained in \cite{HNS16}. 
See also \cite{Kin78, Kor83Ca, Kor83Co, 
LS90Ip, Shi89} 
for other counterexamples related to the convexity framework.

\begin{theorem}\label{thm_counter_ex}
There exists a smooth, concave and symmetric (with respect to the axes) bounded domain $\Omega \subset \R^2$ with $C^{\infty}$-boundary, and a function $f \in C^{\infty}(\R)$, $f>0$ such that
\begin{itemize}
\item[i)] $\sqrt{F}$ is not concave on $(0,+\infty)$,
\item[ii)] $\frac{F}{f}$ is not convex on $(0,+\infty)$,
\end{itemize}
and the problem
$$\begin{cases}
-\Delta u = f(u) & \Omega,
\\
u>0 & \Omega,
\\
u=0 & \partial \Omega,
\end{cases}$$
admits both a quasiconcave solution and a non-quasiconcave solution.
\end{theorem}


The paper is organized as follows. 
\smallskip

In Section \ref{sec_counterex} we discuss the optimality of the conditions in Theorems \ref{thm_BMS} and \ref{thm_corol_main_concave}, by exhibiting a counterexample and proving Theorem \ref{thm_counter_ex}. In Section \ref{sec_quasilin}, instead, we extend Theorem \ref{thm_BMS} to the quasilinear case \eqref{eq_main_intr}, by giving the proof of Theorem \ref{thm_corol_main_concave}. Finally, in Section \ref{sec_ex_rem} we provide some examples and further comments.




\section{Quasi optimality of the assumptions}
\label{sec_counterex}

In this Section we show an example of a semilinear PDE set in a convex, regular (and symmetric) domain with a smooth strictly positive source $f$ not satisfying assumptions $i)$ and $ii)$ of Theorem \ref{thm_BMS}, and which indeed admits a non-quasiconcave solution; incidentally, this problem will admit also a nontrivial quasiconcave solution.

We recall that $u$ is said \emph{quasiconcave} if $u^{-1}((k,+\infty))$ is convex for every $k \in \R$, or equivalently if $u(\lambda x + (1-\lambda) y) \geq \min\{u(x), u(y)\}$ for each $x,y \in \Omega$, $\lambda \in [0,1]$. 
Moreover, $u>0$ is said \emph{$\gamma$-concave} for $\gamma \neq 0$ (resp. $\gamma=0$) if $\sgn(\gamma) u^{\gamma}$ is concave (resp. $\log(u)$ is concave). 

\smallskip

We start with some 
comments on Theorem \ref{thm_BMS}.

In 
\cite{BMS22}, even if not explicitly stated, the authors require 
$f$ to be only \emph{locally} Hölder continuous in Theorem \ref{thm_BMS}: indeed, if $\norm{u}_{\infty} \leq M$, then it is sufficient $f$ to be Hölder continuous in $[0,M]$. More specifically, if $M_f<+\infty$, then $\norm{u}_{\infty} \leq M_f$ by maximum principles \cite[Remark 1.6]{BMS22}; this implies that it is actually sufficient to define $\varphi$ only in $(0, M_f)$. Anyway, if $f\geq 0$, its definition can be extended in $(0,+\infty)$.


Moreover, we highlight that conditions $i)$ and $ii)$ of Theorem \ref{thm_BMS} are not connected: for example,
\begin{itemize}
\item $f(t)=t^p$ with $p \in [0,1]$ satisfies both $i)$ and $ii)$;
\item $f(t)=1+t^p$ with $p\in [0,1]$ satisfies $i)$ but not $ii)$;
\item $f(t)=t^p$ with $p>1$ satisfies $ii)$ but not $i)$;
\item $f(t)=1+t^p$ with $p>1$ does not satisfy either $i)$, neither $ii)$.
\end{itemize}

Finally, we highlight that, in non-power cases, the concaving function $\varphi$ is actually ``less concave'' than power choices of $u$: for example, if 
$$f(t)= \begin{cases}
t^p & \hbox{ for $t \in [0,1]$}, \\
t^q & \hbox{ for $t \in [1,+\infty)$},
\end{cases}
$$
with $0<p<q<1$, then \cite[Theorem 4.2]{Ken85} implies that $\psi(u)=u^{\frac{1-p}{2}}$ is concave, where the exponent is the biggest one given by the Theorem.
On the other hand, Theorem \ref{thm_BMS} implies that $\varphi(u) = \int_1^u F^{-\frac{1}{2}}(s) ds$ is concave, where we have $\varphi(t) \sim t^{\frac{1-q}{2}}$ as $t \to +\infty$.

\smallskip

The example we show in this Section is borrowed by the one developed in \cite{HNS16} 
and it is set in dimension $N=2$. Here we recall the main points of the construction (referring to \cite{HNS16} 
and \cite{Ass23} 
for details), and we make some further comments and suitable adaptations.

\medskip

\textbf{Step 1.}
First, we build a proper family of smooth convex planar domains $\Omega_{\alpha}$, depending on a parameter $\alpha >2$ (which will be definitely taken large enough). The domain $\Omega_{\alpha}$ are ``stadium-like'' shaped (see Figure \ref{fig_sets}) and symmetric with respect both axes.
In particular they can be formally defined by
$$\Omega_{\alpha}:= \big\{ (x,y) \in \R^2 \mid |x| < \alpha + \phi(y), \; |y|<1\big\}$$
where the function $\phi \in C([-1,1]) \cap C^{2}((-1,1))$ satisfies $\phi(\pm 1)=0$, $\phi'(t) \to \mp \infty$ as $t \to \pm 1$, and it can be chosen for example equal to
$$\phi(t):= \lambda \sqrt{1-t^2}, \quad t \in [-1,1],$$
for some $\lambda>0$; in this case $\phi \in C^{\infty}((-1,1))$. 

 \begin{figure}
\begin{center}
\includegraphics[scale=2.5]{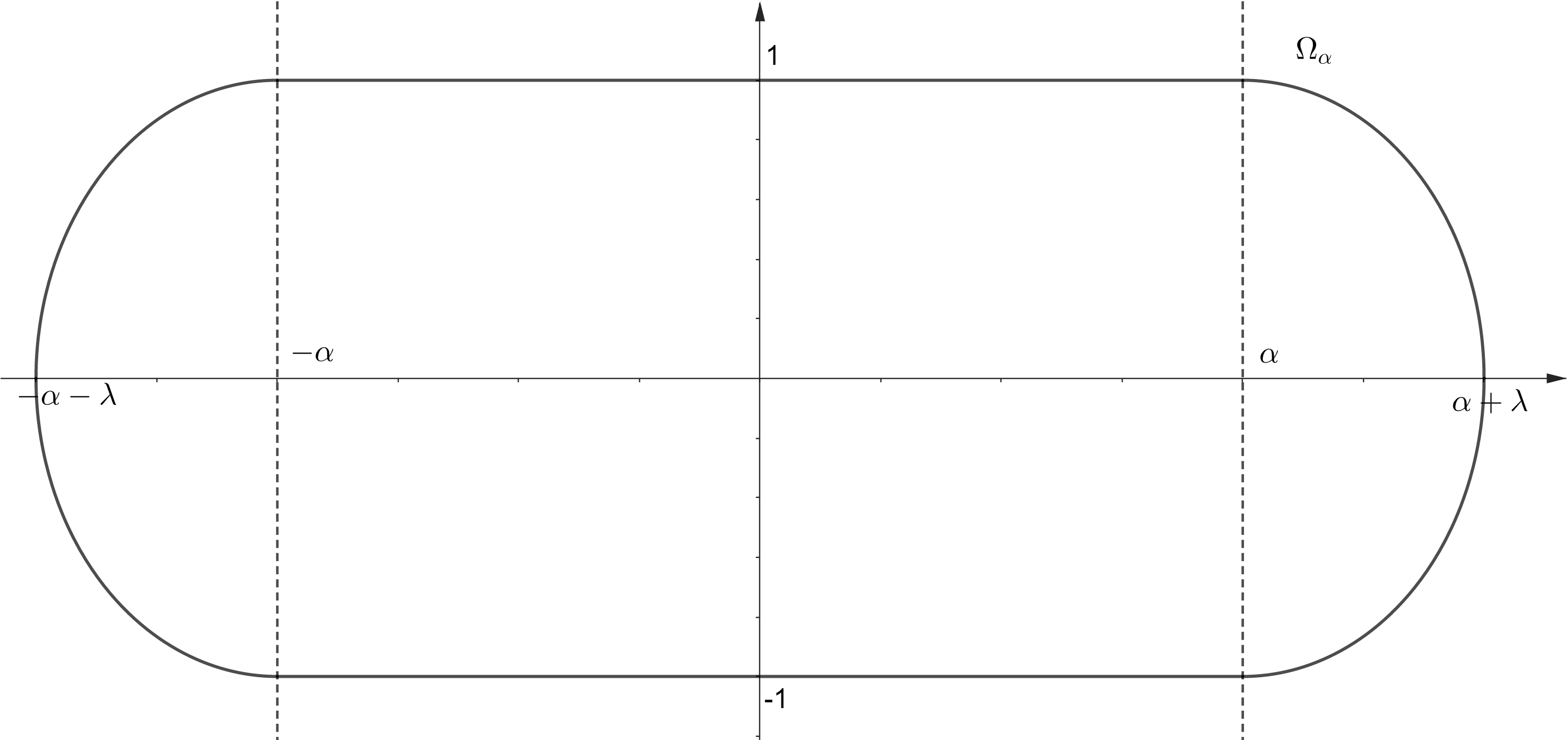}
\end{center}
\small{\caption{
``Stadium-shaped'' sets $\Omega_{\alpha}$.
\label{fig_sets}
}}
\end{figure}

\medskip

\textbf{Step 2.}
We consider the function $\psi \in C^{2}([-1,1])$
$$\psi(y):= \frac{1-y^2}{2}$$
which satisfies $-\psi''=1$ and $\psi(\pm 1)=0$. It results that $\psi$ is the unique minimizer of the functional in $N=1$
$$J: H^1_0((-1,1)) \to \R, \quad J(v):= \int_{-1}^1 \left( \frac{|v'|^2}{2} -v\right) .$$
Trivially extending $\psi$ to $\R^2$ by $\psi(x,y):=\psi(y)$, we also observe that $\psi$ solves the torsion problem
\begin{equation}\label{eq_torsion}
\begin{cases}
-\Delta v = 1 & \Omega_{\alpha},
\\
v = 0 & \partial \Omega_{\alpha}.
\end{cases}
\end{equation}
Moreover, $0 \leq \psi \leq \frac{1}{2}$. 

\medskip

\textbf{Step 3.}
We consider the unique minimizers $v_{\alpha} \in H^1_0(\Omega_{\alpha})$ of the functional in $N=2$
\begin{equation*}
J_{\alpha}: H^1_0(\Omega_{\alpha}) \to \R, \quad 
J_{\alpha}(v):= \int_{\Omega_{\alpha}} \left( \frac{|\nabla v|^2}{2} - v\right) 
\end{equation*}
associated to the torsion problem \eqref{eq_torsion}. It results that $v_{\alpha}$ are classical solutions of the equation, $\frac{1}{2}$-concave, and satisfy
\begin{equation}\label{eq_stima_v_a}
0 < v_{\alpha} < \psi \leq \frac{1}{2}.
\end{equation}

\medskip

\textbf{Step 4.}
We introduce a smooth function $g : \R \to \R$ such that
\begin{equation}\label{eq_propr1_g}
g \equiv 0 \, \hbox{ on $(-\infty, 1]$}, \quad g \equiv 1 \, \hbox{ on $[2, +\infty)$},
\end{equation}
$g$ non-decreasing and
\begin{equation}\label{eq_propr2_g}
\norm{g'}_{\infty} \leq 4.
\end{equation}
This function can be for example chosen in the following way: set $h(t):= e^{-\frac{1}{t}} \chi_{(0,+\infty)}$ we define
$$g(t):= \frac{h(t-1)}{h(t-1)+h(2-t)}, \quad t \in \R.$$
In this case one can straightforwardly check that $g \in C^{\infty}(\R)$ and $\norm{g'}_{\infty} \leq 2.$

\medskip

\textbf{Step 5.}
We consider now the constraint set
$$V_{\alpha}:= \left\{ v \in H^1_0(\Omega_{\alpha}) \mid \int_{\Omega_{\alpha}} g(v) = 1 \right\}$$
and study the functional $J_{\alpha}$ 
restricted to $V_{\alpha}$: one can prove that $J_{\alpha}$ admits a minimizer $u_{\alpha}$, with Lagrange multiplier $\mu_{\alpha}>0$, which in particular satisfies classically the Dirichlet problem
\begin{equation} \label{eq_probl_counter}
\begin{cases}
-\Delta v = 1 + \mu_{\alpha} g'(v) & \Omega_{\alpha},
\\
v > 0 & \Omega_{\alpha},
\\
v = 0 & \partial \Omega_{\alpha},
\end{cases}
\end{equation}
and
$$0 < v_{\alpha} \leq u_{\alpha} < \frac{5}{2}$$
with
\begin{equation}\label{eq_stima2_u_a}
 \norm{u_{\alpha}}_{\infty} >1.
\end{equation}
In addition we have
\begin{equation}\label{eq_lagran_mult}
\mu_{\alpha} > \frac{1}{\norm{g'}_{\infty}} \geq \frac{1}{4}.
\end{equation}
We give some details on \eqref{eq_lagran_mult}.

\begin{proof}[Proof of \eqref{eq_lagran_mult}]
Let $u_{\alpha}$ be fixed and consider $w_{\alpha}$ the classical solution of the linear problem
\begin{equation} 
\begin{cases}
-\Delta v = 1 + \frac{g'(u_{\alpha})}{\norm{g'}_{\infty}} & \Omega_{\alpha},
\\
v = 0 & \partial \Omega_{\alpha}.
\end{cases}
\end{equation}
Observed that 
$$ -\Delta w_{\alpha} = 1 + \frac{g'(u_{\alpha})}{\norm{g'}_{\infty}} \leq 2 = - \Delta (2 v_{\alpha})$$
we obtain, by the Comparison Principle, that $w_{\alpha} \leq 2 v_{\alpha}.$
If by contradiction $\mu_{\alpha} \leq \frac{1}{\norm{g'}_{\infty}}$, then
$$ -\Delta w_{\alpha} = 1 + \frac{g'(u_{\alpha})}{\norm{g'}_{\infty}} \geq 1+ \mu_{\alpha} g'(u_{\alpha}) = -\Delta u_{\alpha}$$
and again by the Comparison Principle $w_{\alpha} \geq u_{\alpha}$. Thus
$$ u_{\alpha} \leq w_{\alpha} \leq 2 v_{\alpha}.$$
On the other hand, exploiting \eqref{eq_stima2_u_a} and \eqref{eq_stima_v_a} we obtain, for some point $P\in \Omega_{\alpha}$,
$$1<u_{\alpha}(P) \leq w_{\alpha}(P) \leq 2 v_{\alpha}(P) <1,$$
getting a contradiction.
\end{proof}

\smallskip

\textbf{Step 6.}
Consider the level set
$$\omega_{\alpha}:=u_{\alpha}^{-1}((1,+\infty)).$$
One can show that $\omega_{\alpha}$ is symmetric with respect to both axes and fulfills the \emph{rectangle property}, i.e.
$$ \forall (\bar{x}, \bar{y}) \in \omega_{\alpha} \, : \, (-\bar{x}, \bar{x}) \times (-\bar{y}, \bar{y}) \subset \omega_{\alpha};$$
in particular, $\omega_{\alpha}$ is star-shaped with respect to the origin. Moreover, $\omega_{\alpha} \subset \Omega_{\alpha}$ are not \emph{too elongated} nor \emph{too thin}, that is there exist two $\alpha$-independent constants $C_1, C_2>0$ such that
$$0 < \sup_{(x,y) \in \omega_{\alpha}} |x| < C_1, \quad \hbox{and} \quad 0 < C_2 < \sup_{(x,y) \in \omega_{\alpha}} |y|,$$
for each $\alpha>2.$

\medskip

\textbf{Step 7.}
We show that, up to choosing $\alpha >2$ sufficiently large, one can make $u_{\alpha}$ arbitrary similar to the function $\psi$: namely, for every small $\eps>0$, there exists an $\alpha_0=\alpha_0(\eps)\gg 0$ such that, for any $\alpha>\alpha_0$ we have
$$|u_{\alpha}(x,y)-\psi(y)| < \eps \quad \hbox{for every $(x,y) \in \mc{S}_{\alpha}:=(k_0, \alpha - k_0) \times (-1,1) \subset \Omega_{\alpha}$}, $$
where $k_0 \in (0,\frac{\alpha_0}{4})$ is suitably chosen.

\medskip

\textbf{Step 8.}
We first observe that $v_{\alpha}>0$ are trivially solutions (for any $\alpha$) of problem \eqref{eq_probl_counter}, since $\norm{v_{\alpha}}\leq \frac{1}{2}<1$ by \eqref{eq_stima_v_a} and $g'_{|(-\infty,1)} \equiv 0$ by \eqref{eq_propr1_g}. Moreover $v_{\alpha}$ is $\frac{1}{2}$-concave and thus, in particular, problem \eqref{eq_probl_counter} admits a quasiconcave solution.

We show instead that the built solutions $u_{\alpha}$ are not quasiconcave: indeed, chosen $\eps$ small and 
$$\eta \in \Big(\frac{1}{2}- \frac{C_2^2}{8} + \eps, \frac{1}{2}-\eps\Big) \subset (0,1)$$
we pick three points $P$, $Q_{\alpha}$ and $M_{\alpha}$ (see Figure \ref{fig_three_points}) such that
$$P:= (0, C_2) \in \omega_{\alpha}, \quad Q_{\alpha}:= \Big (\frac{\alpha}{2}, 0\Big) \in \mc{S}_{\alpha}, \quad M_{\alpha} := \frac{P+Q_{\alpha}}{2}
\in \mc{S}_{\alpha}$$
which imply (roughly), for $\alpha \gg 0$,
$$u_{\alpha}(Q_{\alpha}) \approx \psi(Q_{\alpha}) = \frac{1}{2} > \eta > \frac{1}{2}-\frac{C_2^2}{8} = \psi(M_{\alpha}) \approx u_{\alpha}(M_{\alpha}).$$
So one has shown that 
$$P, Q_{\alpha} \in u_{\alpha}^{-1}((\eta, +\infty)), \quad \hbox{but} \quad M_{\alpha} \notin u_{\alpha}^{-1}((\eta, +\infty)).$$
Being $u_{\alpha}^{-1}((\eta, +\infty))$ not convex, we have that $u_{\alpha}$ is a solution of \eqref{eq_probl_counter} which is not quasiconcave.

\begin{figure}
\begin{center}
\includegraphics[scale=1.9]{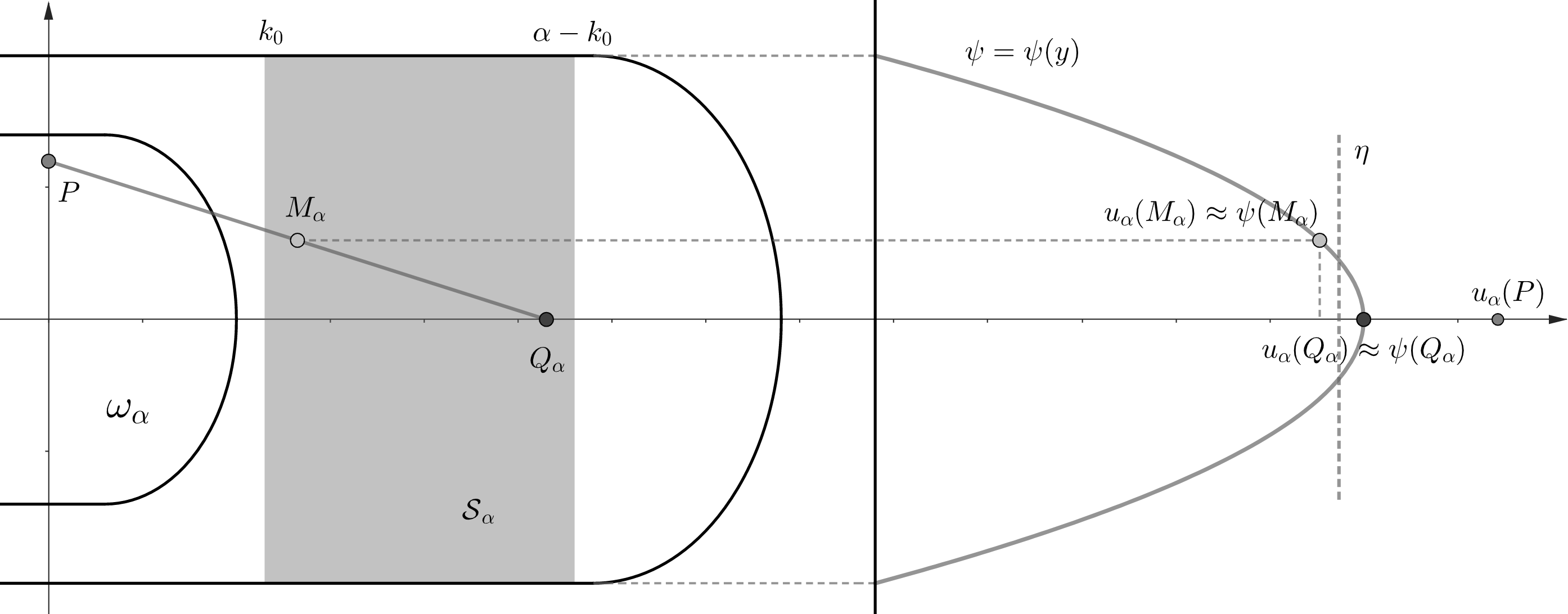}
\end{center}
\small{\caption{
Visual proof of the non-convexity of $u^{-1}((\eta,+\infty))$: indeed, \\ $u_{\alpha}(P)>1>\eta$ and $u(Q_{\alpha})\approx \frac{1}{2}>\eta$, while $u_{\alpha}(\frac{P+Q_{\alpha}}{2})\approx \frac{1}{2}-\frac{C_2^2}{8} <\eta$.
\label{fig_three_points}
}}
\end{figure}

\medskip

We are now ready to conclude Theorem \ref{thm_counter_ex}.

\begin{proof}[Proof of Theorem \ref{thm_counter_ex}]
What remains to show is that the function ($\alpha$ is now fixed, we write $\mu:= \mu_{\alpha}$)
$$f(t):=1+\mu g'(t), \quad t \in \R$$
does not satisfy $i)$ and $ii)$ of Theorem \ref{thm_BMS}. Notice first that $f\in C^{\infty}(\R)\subset {\rm Lip}_{\rm loc}(\R)$ and set $F(t):= \int_0^t f(\tau) d \tau = t + \mu g(t).$

\smallskip

\emph{$\bullet$ $\sqrt{F}$ is not concave.}
Observe that
$$\sqrt{F(t)} \equiv \sqrt{t} \, \hbox{ for $t \in (-\infty, 1]$}, \quad \sqrt{F(t)} \equiv \sqrt{t + \mu} \, \hbox{ for $t \in [2, +\infty)$},$$
and that the tangent line to $\sqrt{F}$ in $t=1$ is given by $y(t) = \frac{t+1}{2}$.
If by contradiction $\sqrt{F}$ were concave, then we would have
$$\sqrt{F(t)} \leq \frac{t+1}{2} \quad \hbox{for each $t \in \R$}$$
and in particular for $t=2$, which means
$$\sqrt{2+\mu} \leq \frac{3}{2},$$
i.e. $\mu \leq \frac{1}{4}$, in contradiction with \eqref{eq_lagran_mult}.

\smallskip

\emph{$\bullet$ $\frac{F}{f}$ is not convex.}
Since
$$\frac{F(t)}{f(t)} \equiv t \, \hbox{ for $t \in (-\infty, 1]$}, \quad \frac{F(t)}{f(t)} \equiv t+ \mu \, \hbox{ for $t \in [2, +\infty)$},$$
and since the tangent line to $\frac{F}{f}$ in $t=1$ is $y(t) = t + \mu$,
if we assume by contradiction $\frac{F}{f}$ convex, then
$$\frac{F(t)}{f(t)} \geq t+\mu \quad \hbox{for each $t \in \R$};$$
in particular, for $t=1$,
$$1 \geq 1+\mu,$$
impossible, since $\mu>0$. We can see the contradiction also exploiting \cite[Remark 1.5]{BMS22} (see Lemma \ref{lem_corol_BMS} below): in order to have $\frac{F}{f}$ convex we would need $\gamma=\frac{F f'}{f^2} = \frac{ (t+\mu g)(\mu g'')}{(1+\mu g')^2}$ non-increasing; but this is clearly not possible, since $\gamma(t) \equiv 0$ for $t \in (-\infty, 1] \cup [2, +\infty)$, and not identically zero in between (actually, it changes sign).
\end{proof}

We end this Section with some comments.

\begin{remark}
We showed the counterexample in dimension $N=2$, for $p=2$ and $a\equiv 1$. While the generalization to $p>1$ and $a\nequiv 1$ seems to create no particular difficulty, it would be instead interesting to show a counterexample also in dimension $N\geq 3$. We indeed highlight that, when $N=1$, then all the convex sets are balls, thus by \cite{GNN79} 
the solutions are automatically decreasing (and positive), thus quasiconcave.

A second goal is to show two counterexamples which deny assumptions $i)$ (but not $ii)$) and $ii)$ (but not $i)$) respectively, in Theorem \ref{thm_BMS}. 
To this regard we observe that, in order to deny $ii)$ it was sufficient \eqref{eq_propr1_g}, while to deny $i)$ we needed the extra assumption on $g$ \eqref{eq_propr2_g}.

We further notice that the built $\sqrt{F}$ (resp. $\frac{F}{f}$) is definitely concave (resp. convex) at the origin and at infinity, and this shows that such properties cannot be relaxed in this sense.
\end{remark}


\section{Modified quasilinear Schrödinger equations}
\label{sec_quasilin}

We move now to the study of the quasilinear equation \eqref{eq_main_intr}, namely 
\begin{equation}\label{eq_main_sect}
\begin{cases}
-\dive\big(a(u) \nabla u\big) + \frac{a'(u)}{2} |\nabla u|^2 = f(u) & \Omega,
\\
u>0 & \Omega,
\\
u=0 & \partial \Omega.
\end{cases}
\end{equation}


We deal first with some properties of the change of variables which allows to transform \eqref{eq_main_sect} into a semilinear problem; they can be borrowed by the ones in \cite{GS13}, 
with some additional standard arguments (see \cite{Ass23} for details).

\begin{proposition}
Let $a \in C^1([0,+\infty))$ 
and assume the strict ellipticity of the problem, i.e. there exists $\nu>0$ such that
\begin{equation}
	\label{eq_ellipticity}
a(t) \geq \nu >0, \quad \hbox{ for any $t \in \R$}.
\end{equation}
Consider the ODE Cauchy problem
\begin{equation}\label{eq_Cauchy_prob}
\begin{cases}
g'=\frac{1}{\sqrt{a(g)}} & (0,+\infty),
\\
g(0)=0, &
\end{cases}
\end{equation}
then the uniquely determined $g$ 
is a riparametrization of the identity, i.e. $g \in C^2([0,+\infty)) \cap {\rm Lip}([0,+\infty))$, it is strictly increasing, invertible, with $g([0,+\infty))=[0,+\infty)$ (and the inverse $g^{-1}$ shares the same properties).

Let in addition $f \in C
([0,+\infty))$. Then $u \in C(\overline{\Omega}) \cap C^2(\Omega)$ is a classical solution of the quasilinear equation \eqref{eq_main_sect} if and only if
$$v:=g^{-1}(u), $$
$v \in C(\overline{\Omega}) \cap C^2(\Omega)$, is a solution of the semilinear equation
\begin{equation}\label{eq_change_var}
\begin{cases}
-\Delta v = \left(\frac{f}{\sqrt{a}}\right)(g(v))=:h(v) & \Omega,
\\
v>0 & \Omega,
\\
v=0 & \partial \Omega.
\end{cases}
\end{equation}
\end{proposition}

Set now
$$M_h:=\inf \big\{ t>0 \mid h(t) \leq 0\big\}$$
it is easy to show that
$$M_h = g(M_f)$$
with the convention $g(+\infty)=+\infty$.

\smallskip

In order to exploit concavity results related to \eqref{eq_change_var}, we first rewrite Theorem \ref{thm_BMS} in the semilinear case with $f \in C^1$ (see \cite[Remark 1.5]{BMS22}).

\begin{lemma}\label{lem_corol_BMS}
Let $N\geq 1$, $\Omega \subset \R^N$ be a bounded, convex domain with $C^2$-boundary, and let $h\in C^{\sigma}_{\rm loc}([0,+\infty))\cap C^1((0,+\infty))$ for some $\sigma \in (0,1]$. 
Let $v \in C(\overline{\Omega}) \cap C^2(\Omega)$ 
be a classical solution of
$$\begin{cases}
-\Delta v = h(v) & \Omega,
\\
v>0 & \Omega,
\\
v=0 & \partial \Omega.
\end{cases}$$
 Assume moreover that 
\begin{itemize}
\item $\frac{H h'}{h^2}$ is non-increasing in $(0,M_h)$ with $\lim_{t \to 0^+} \frac{H h'}{h^2}\leq \frac{1}{2}$
\end{itemize}
where $M_h:=\inf \{ t>0 \mid h(t) \leq 0\}>0$, $H(t):= \int_0^t h(\tau) d\tau$ and $h >0$ on $(0,M_h)$. 
Then $\psi(v)$ is concave, where 
$$ \psi(t):= \int_1^t \frac{1}{\sqrt{H(s)}} ds, \quad t \in (0,+\infty).$$
\end{lemma}

\begin{theorem} \label{thm_main_concave}
Let $N\geq 1$, $\Omega \subset \R^N$ be a bounded, convex domain with $C^2$-boundary, $a\in C^1([0,+\infty))$ 
 satisfying \eqref{eq_ellipticity}, and let $f\in C^{\sigma}_{\rm loc}([0,+\infty)) \cap C^1((0,+\infty))$ for some $\sigma \in (0,1]$. 
Let $u \in C(\overline{\Omega}) \cap C^2(\Omega)$ 
be a classical solution of \eqref{eq_main_sect}. Set
\begin{equation}\label{eq_def_gamma}
\gamma(t) := \frac{F(t) f'(t)}{f^2(t)} - \frac{1}{2} \frac{F(t) a'(t)}{ f(t) a(t)}, \quad t \in (0, M_f),
\end{equation}
where $M_f:=\inf \{ t>0 \mid f(t) \leq 0\}>0$ and $F(t):= \int_0^t f(\tau) d\tau.$
Assume moreover $f >0$ on $(0,M_f)$ and 
\begin{itemize}
\item $\gamma$ is non-increasing in $(0,M_f)$ with $\lim_{t \to 0^+}\gamma(t)\leq \frac{1}{2}$.
\end{itemize}
Then $\varphi(u)$ is concave, where 
\begin{equation}\label{eq_def_varphi}
 \varphi(t):= \int_{\mu}^t \sqrt{\frac{a(s)}{F(s)}} ds, \quad t \in (0,+\infty),
 \end{equation}
$\mu:= \nu^{-\frac{1}{2}}$. In particular, $u$ is quasiconcave.
\end{theorem}

\begin{proof}
Let us consider the modified problem \eqref{eq_change_var}. First of all, we observe that $h$ is locally Hölder continuous: indeed, for any $K \subset \subset [0,+\infty)$ and any $t, s \in K$, set for simplicity $x:=g(t)$, $y:=g(s)$ and $\theta:= \min\{\sigma, \frac{1}{2}\}$, we have 
\begin{align*}
|h(t)-h(s)| &\leq \frac{ |f(x)\sqrt{a(y)} - f(y) \sqrt{a(x)}|}{|\sqrt{a(x) a(y)}|} \leq \frac{1}{\nu} \left( \sqrt{a(y)} |f(x)-f(y)| + |f(y)| |\sqrt{a(x)}-\sqrt{a(y)}|\right) \\
& \lesssim \left( \norm{\sqrt{a\circ g}}_{L^{\infty}(K)} + \norm{f\circ g}_{L^{\infty}(K)} \right) |g(t)-g(s)|^{\theta} \leq C(a,f,g,K) |t-s|^{\theta},
\end{align*}
where we have used that $a\in C^1([0,+\infty))$ (and thus $\sqrt{a}$ is $\frac{1}{2}$-Hölderian) and $g$ is Lipschitz. Hence $h$ is locally $\theta$-Hölder continuous; notice that, far from $t=0$, we have $h \in C^1$, thus here $h$ is locally Lipschitz continuous. 

Then, we define the antiderivative of $h$
$$H(t):= \int_0^t h(\tau) d \tau = \int_0^t \left(\frac{f}{\sqrt{a}}\right)(g(\tau)) d\tau.$$
By a change of variable, and exploiting the ODE \eqref{eq_Cauchy_prob} we obtain
$$H= F \circ g.$$
Thus, by looking at Lemma \ref{lem_corol_BMS} we evaluate $\frac{H h'}{h^2}$, which again by \eqref{eq_Cauchy_prob} gives
$$ \frac{H h'}{h^2} = \gamma \circ g.$$
In particular, being $g$ increasing and continuous with $g(0)=0$ we obtain
$$ \frac{H h'}{h^2} \hbox{ is non-increasing} \iff \gamma \hbox{ is non-increasing},$$
$$\lim_{t \to 0^+} \left(\frac{H h'}{h^2}\right)(t) \equiv \lim_{t \to 0^+} \gamma(t).$$
Applying Lemma \ref{lem_corol_BMS} we gain that $\psi(v)$ is concave, where $v=g^{-1}(u)$ and
$$ \psi(t)= \int_1^t \frac{1}{\sqrt{H(s)}} ds = \int_1^t \frac{1}{\sqrt{F(g(s))}} ds = \int_{g(1)}^{g(t)} \sqrt{\frac{a(s)}{F(s)}} ds = \varphi(g(t)) + \int_{g(1)}^{\mu} \sqrt{\frac{a(s)}{F(s)}} ds$$
where we used again a change of variable and \eqref{eq_Cauchy_prob}; we observe $g(1) \leq \mu$ since
$$g(1) = \int_0^1 g'(\tau) d\tau = \int_0^1 \frac{1}{\sqrt{a(g(\tau))}} d \tau \leq \nu^{-\frac{1}{2}}=\mu.$$
Set $C:=\int_{g(1)}^{\mu} \sqrt{\frac{a(s)}{F(s)}} ds$ we get that the function
$$\varphi(u)=\varphi(g(v))-C=\psi(v)-C$$
is concave, which is the claim.
\end{proof}


We rephrase Theorem \ref{thm_main_concave} in a less general framework, but with assumptions on $f$ and $a$ more strictly related to the original ones of Theorem \ref{thm_BMS}; namely, we prove Theorem \ref{thm_corol_main_concave}.

\begin{proof}[Proof of Theorem \ref{thm_corol_main_concave}]
Observed that, by \eqref{eq_def_gamma}, 
$$\gamma(t) = \frac{F(t) f'(t)}{f^2(t)} - \frac{1}{2}\xi(t),$$
and that $i)$ and $ii)$ imply that $\frac{F f'}{f^2}$ is non-increasing in $(0,M_f)$ with $\lim_{t \to 0^+} \frac{F f'}{f^2}\leq \frac{1}{2}$, we achieve the claim by the assumption $iii)$ and Theorem \ref{thm_main_concave}.
\end{proof}

\begin{remark}
We observe that the conditions stated in Theorem \ref{thm_main_concave} and Theorem \ref{thm_corol_main_concave} are scaling invariant, that is, they do not change by substituting $f$ and $a$ with $\lambda f$ and $\mu a$, where $\lambda, \mu \in (0,+\infty)$. 
Moreover, the function $\varphi$ does not depend on the particular domain $\Omega$; it could be interesting to investigate the existence of a natural function $\varphi$ modeled also on the shape (in particular, the curvature) of the domain.
\end{remark}


\section{Examples and remarks}
\label{sec_ex_rem}

We start by observing that, under mild assumptions on $f$ and $\Omega$, solutions of PDEs are never concave.

\begin{remark}\label{rem_never_concav}
Let $f:[0,+\infty) \to [0,+\infty)$ with $f(0)= 0$, and consider the problem 
$$\begin{cases}
-\Delta u = f(u) & \Omega,
\\
u>0 & \Omega,
\\
u=0 & \partial \Omega,
\end{cases}$$
and a classical solution $u \in C^2 (\overline{\Omega})$ (e.g. one may assume $\partial \Omega \in C^{2,\sigma}$ and $f \in C^{\sigma}(\R)$ for some $\sigma \in (0,1]$). Let $P \in \partial \Omega$ and assume $\partial \Omega$ parametrized, near $P$ (say, in a neighborhood $U$), by a $C^2$-curve $x_N=\phi(\tilde{x})$, $\tilde{x}=(x_1, \dots x_{N-1})$, with 
\begin{equation}\label{eq_contex_Korev}
\nabla \phi(\tilde{P})=0, \quad \Delta \phi(\tilde{P})>0;
\end{equation}
in particular the last condition means that the mean curvature of $\partial \Omega$ in $P$ is strictly positive\footnote{Indeed, $\kappa(P) = \frac{1}{2} \dive \left ( \frac{\nabla \phi}{\sqrt{1+|\nabla \phi|^2}}\right)(\tilde{P}) = \frac{1}{2} \Delta \phi(\tilde{P})$ when $\nabla \phi (\tilde{P}) =0$. Notice that, if $\Omega \subset \R^N$ is smooth, convex and bounded, the existence of such a point is always ensured, for example by the Minkowski inequality $\int_{\partial \Omega} \kappa d\sigma \gtrsim |\partial \Omega|^{\frac{N-2}{N-1}}$ (when $N\geq 3$), or by a comparison argument with an external shrinking ball.}.
 Here we have fixed the axis coherent with the local parametrization, observing that the equation is translation and rotation invariant.
Set
$$\Phi(\tilde{x}):=u(\tilde{x}, \phi(\tilde{x})) \equiv 0 $$
for every $(\tilde{x}, \phi(\tilde{x})) \in \partial \Omega \cap U$. Thus we have
$$0 \equiv \Phi_{ii}(\tilde{x}) = u_i(\tilde{x}, \phi(\tilde{x})) + 2 u_{iN}(\tilde{x}, \phi(\tilde{x})) \phi_i(\tilde{x}) + u_{NN}(\tilde{x}, \phi(\tilde{x})) \phi_i^2(\tilde{x}) + u_N(\tilde{x}, \phi(\tilde{x})) \phi_{ii}(\tilde{x})$$
for $i=1, \dots, N-1$. In particular, by \eqref{eq_contex_Korev}
$$0 = u_i(P) + u_N(P) \phi_{ii}(\tilde{P}), \quad i= 1, \dots, N-1$$
and hence, observed that $\Delta u(P) = -f(u(P))=-f(0)=0$, 
$$0 = -u_{NN}(P) + u_N(P) \Delta \phi(\tilde{P}).$$
Since $f \geq 0$, by the Hopf's Lemma we have $ u_N(P) > 0$, thus, by \eqref{eq_contex_Korev}, 
$$u_{NN}(P) = u_N(P) \Delta \phi(\tilde{P}) > 0.$$
This automatically rules out the concavity of $u$, since if $u$ were concave then the second derivative in every direction would be nonpositive.

This remark in particular applies to the eigenvalue problem \cite{Kor83Co}. 
The condition $f(0)=0$ cannot be removed, as shown by the torsion problem where concave solutions may exist \cite{HNST18, Kos87}: for example, if $u$ is a solution of the torsion problem in $\Omega$ with a strict maximum $u(\bar{x})=:c$, then $v:= u - (c-\eps)$ is a concave solution in the set $\Omega_c:=u^{-1}( (c-\eps, c])$ with $\eps$ sufficiently small.

When $\Omega$ has only points with flat curvature (whenever defined), we can still show that solutions of the first eigenvalue problem are not concave. For example, when $N=2$ and $\Omega=(0,\pi) \times (0,\pi)$, it is a straightforward computation seeing that the solution $u(x,y)=\sin(x) \sin(y)$ has nonpositive eigenvalues only in the $1$-norm ball $B_{\frac{\pi}{2}}(\frac{\pi}{2},\frac{\pi}{2})$, i.e. far from the corners; when $N\geq 3$ and $\Omega$ is a cube, by separation of variables we can apply the previous argument on cross sections. This non-concavity extends to (planar) first eigenfunctions of arbitrary convex polygons since, near a corner with amplitude $\frac{\pi}{b}$, we have by some barrier argument (see also \cite[Remark 1.6]{Ind22}), in radial coordinates centered in the corner, 
$$u(\rho, \phi) = C_b \rho^b \sin(b \phi) + o(\rho^b)$$
for some $C_b>0$, thus (being $b>1$) $u$ cannot be concave. 

\end{remark}

We show now the applicability of Theorem \ref{thm_corol_main_concave} to some particular cases.

\begin{figure}
\begin{center}
\includegraphics[scale=0.38]{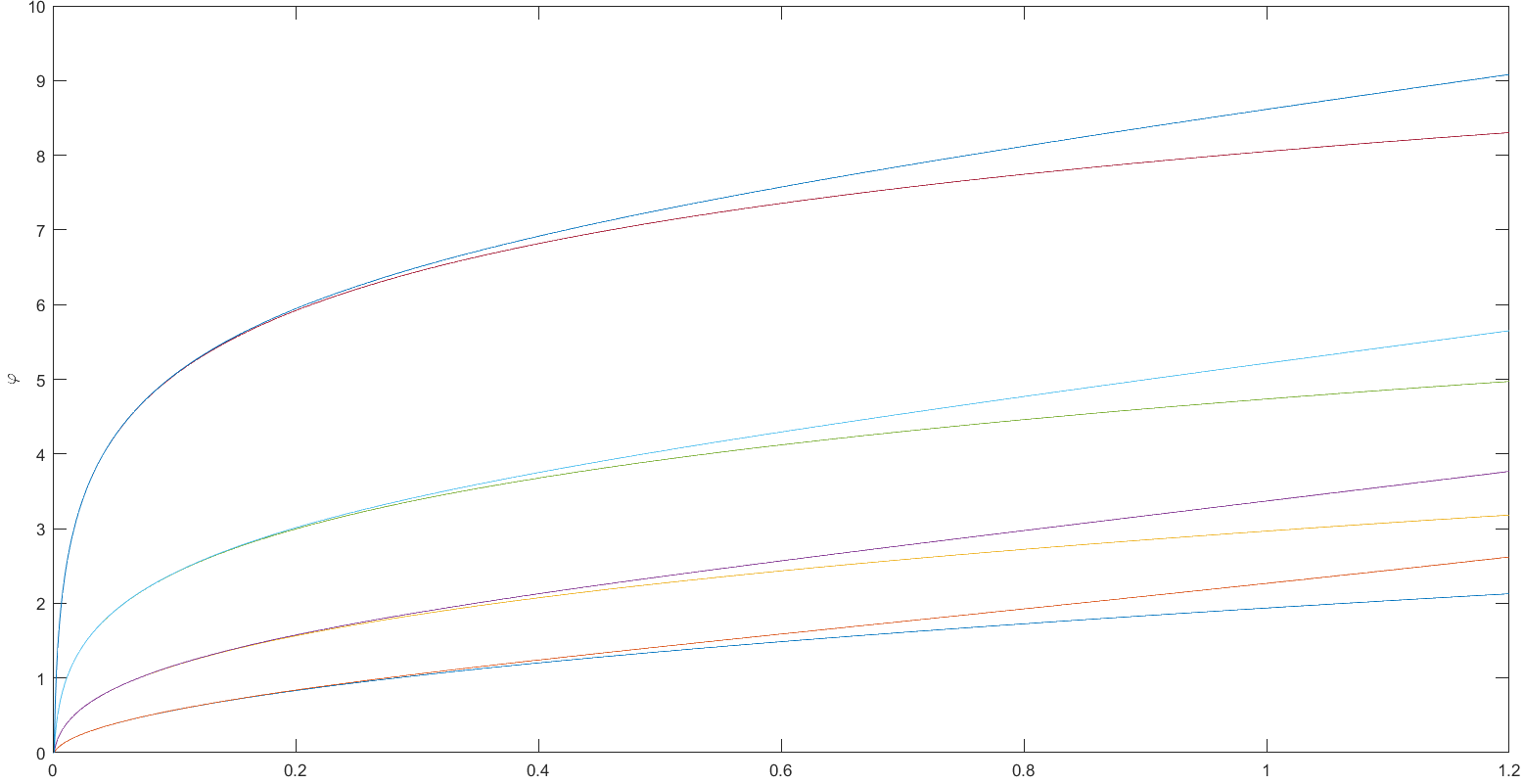}
\end{center}
\small{\caption{
Numerical computations of $\varphi$, comparing the cases $a(t)\equiv 1$ and $a(t)=1+2t^2$, for different values of $f(t)=t^q$: $q=0,\, 0.3,\, 0.6,\,0.9$. It is evident that the behavior at the origin is dictated only by the source $f$.
\label{fig_graf_a}
}}
\end{figure}

\begin{example}\label{ex_fisico_power}
Consider (see \eqref{eq_intr_MNLS}) 
$$a(t)=1+2t^2, \quad f(t)=t^q;$$
in this case $F(t)= \frac{t^{q+1}}{q+1} $, $M_f=+\infty$ and conditions $i)$-$ii)$ are fulfilled whenever $t^{\frac{q+1}{2}}$ is concave, i.e. $q \in [0,1]$. As regards $iii)$ instead, we have (up to multiplicative constants)
$$\xi(t)= \frac{t^2}{1+2t^2}$$
which is positive 
and increasing on $(0,+\infty)$, whatever $q$ is. Thus, set (notice $\nu=1$)
$$\varphi(t):= \int_1^t \sqrt{\frac{1+2s^2}{s^{q+1}}} ds$$
we have $\varphi(u)$ concave, for every classical solution of 
$$\begin{cases}
-\Delta u - u \Delta(u^2)=u^q & \Omega,
\\
u>0 & \Omega,
\\
u=0 & \partial \Omega.
\end{cases}$$
Notice that
$$\varphi(t) \sim \int_1^t s^{-\frac{q+1}{2}} ds \sim t^{\frac{1-q}{2}} - 1 \quad \hbox{as $t\to 0$} $$
(additive and multiplicative constants play no role in concavity arguments). We may wonder if, when $q<1$, we can apply results in \cite{BS13} 
(see also \cite{Kor83Ca, Kor83Co, Ken85}) 
to get the concavity of $u^{\frac{1-q}{2}}$: to achieve this, we should have (see \eqref{eq_int_old_result} with $a$, $f$ as above and $\gamma=\frac{1-q}{2}$)
$$t \in [0,+\infty) \mapsto \frac{t(\beta t^{\frac{1}{1-q}} +1)}{1+2 t^{\frac{1}{1-q}}}$$
concave for every $\beta \geq 0$. A straightforward computation shows that, set $\theta:=\frac{1}{1-q}>1$, the second derivative is given by
$$ t \in [0,+\infty) \mapsto (2-\beta) \frac{2\theta (\theta-1)t^{\theta-1}}{(2t^{\theta}+1)^3} \left(t^{\theta}- \frac{\theta+1}{2(\theta-1)}\right) $$
which actually does not fulfill the request. This argument shows that, maybe, power functions are not the right choice to deal with quasilinear equations of the type \eqref{eq_main_intr}, while a suitable $\varphi$ as the above one gives a natural example of function which makes a solution concave.
This seems to highlight a difference between the quasilinear MNLS case and the quasilinear $p$-Laplacian case, where suitable powers of solutions are allowed to be concave \cite{Sak87}.
%
\end{example}

%




\begin{example}
We choose
$$a(t)=1+2 t^2, \quad f(t)=1-t.$$
In this case $M_f = 1$ and
 $F(t)= t- \frac{t^2}{2}$. 
 Straightforward computations show the applicability of Theorem \ref{thm_corol_main_concave}.
 %
\end{example}

\begin{example}
Let $\lambda \in [0,1)$ and (see \eqref{eq_intr_relSch})
$$a(t) = 1 \pm \frac{t^2}{2(1+t^2)}, \quad f(t) = -\lambda t + \frac{t}{\sqrt{1+t^2}} .$$
Here $M_f = \frac{\sqrt{1-\lambda^2}}{\lambda} \in (0, +\infty]$ and $F(t)= - \lambda \frac{t^2}{2} + \sqrt{1+t^2}-1$. Numerical computation suggest that $\gamma$ defined in \eqref{eq_def_gamma} equals $\frac{1}{2}$ in $t=0$ and then decreases. 
In particular, solutions to
$$\begin{cases}
-\Delta u + \left( \lambda + \frac{1 \mp \frac{1}{2} \Delta \sqrt{1+u^2}}{\sqrt{1+u^2}}\right) u = 0 & \Omega,
\\
u>0 & \Omega,
\\
u=0 & \partial \Omega,
\end{cases}$$
are such that $\varphi(u)$ is concave, where $\varphi$ is given by \eqref{eq_def_varphi} (and $\mu= \sqrt{2}$).
\end{example}

%

\begin{remark}
\label{rem_general_explod}
We highlight that one can relax the assumptions of Theorem \ref{thm_corol_main_concave} to $a \in C^{\omega}_{\rm loc}([0,+\infty)) \cap C^1((0,+\infty))$ for some $\omega \in (0,1]$; in this case $h$, defined in \eqref{eq_change_var}, verifies $h \in C^{\theta}_{\rm loc}([0,+\infty))$ with $\theta:= \min\{\sigma, \frac{\omega}{2}\}$.

Moreover, we observe that Theorem \ref{thm_corol_main_concave} can be proved also in the case 
$$a \in C^1([0, \beta)), \quad \lim_{t \to \beta^{-}} a(t) = +\infty,$$
and $f \in C^{\sigma}_{\rm loc}([0,\beta))\cap C^1((0,\beta))$ for some $\sigma \in (0,1]$ and $\beta \in (0,+\infty)$. 
Some slight changes in the proofs occur.
\begin{itemize}
\item The transformation $g:[0,+\infty) \to \R$ given by \eqref{eq_Cauchy_prob}, still regular and strictly increasing, verifies $g(t) \to \beta$ as $t \to +\infty$. In particular, $g^{-1}: [0, \beta) \to [0,+\infty)$. Here we exploit that $a$ explodes in $\beta$.
\item The classical solutions of the Dirichlet boundary problem \eqref{eq_main_sect} satisfy $0<u(x)<\beta$, $x \in \Omega$. This can be seen for example by writing
$$-\Delta u = \frac{a'(u)}{2 a(u)} |\nabla u|^2 + \frac{f(u)}{a(u)};$$
if there exists an $x \in \Omega$ such that $u(x)=\beta$, then -- observed that\footnote{Otherwise, $\log(a(t))$ would be controlled by some $C t$, and thus bounded on compact sets.} 
$\pabs{\frac{a'(t)}{a(t)}}$ cannot be bounded near $t=\beta$ --
we get a contradiction with the definition of classical solution. 
In particular, being continuous on $\overline{\Omega}$, we have that each solution $u$ is far from $\beta$, i.e. $\norm{u}_{\infty} < \beta$, and we can work with $f \in C^1((0, \norm{u}_{\infty}])$.

Notice instead that solutions to \eqref{eq_change_var} are not subject to any restriction, since $a \circ g$ is defined on the whole $\R$, and so it is the source $h= \frac{f}{\sqrt{a}} \circ g$.
\end{itemize}
The remaining part of the proof follows the lines of the original one.
\end{remark}

\begin{example}
Thanks to Remark \ref{rem_general_explod} (setting $\beta=1$) we can consider for $\lambda \in [0,1)$ (see \eqref{eq_intr_Heisenb}),
$$a(t) = 1 + \frac{t^2}{2(1-t^2)}, \quad f(t)= t - \lambda \frac{t}{\sqrt{1-t^2}};$$
here $M_f = \sqrt{1-\lambda^2} \in [0,1)$, $F(t) = \frac{t^2}{2} + \lambda ( \sqrt{1-t^2}-1)$. Again, numerical computation shows that $\gamma(0)=\frac{1}{2}$ and $\gamma$ decreasing in $(0,M_f)$, 
where $\gamma$ is defined in \eqref{eq_def_gamma}. Thus, again, $\varphi(u)$ is concave, where $\varphi$ is given by \eqref{eq_def_varphi} ($\mu=1$) and $u$ satisfies
$$\begin{cases}
-\Delta u + \frac{1}{2} u\frac{\Delta\sqrt{1-u^2}}{\sqrt{1-u^2}} = u - \lambda \frac{u}{\sqrt{1-u^2}}& \Omega,
\\
u>0 & \Omega,
\\
u=0 & \partial \Omega.
\end{cases}$$
The case $\lambda=0$ can be easily (and explicitly) generalized also to powers $f(t)=t^q$, with $q \in [0,1]$.
\end{example}

\begin{remark}
We remark that, without assuming $f\in C^1$, it is not straightforward to exploit the ``size of the convexity'' of the involved functions and reach the desired claim of Theorem \ref{thm_corol_main_concave}.
 Indeed, consider the notations of Theorem \ref{thm_main_concave}: to conclude that $\sqrt{H} = \sqrt{F}\circ g$ is concave, one can assume $\sqrt{F}$ concave and $a$ non-decreasing (i.e. the physical case, which implies $g$ concave), being $\sqrt{H}$ composition of a concave non-decreasing function and a concave function. On the other hand, to deduce $\frac{H}{h} = \left( \frac{F}{f} \sqrt{a} \right) \circ g$ convex, we need $\frac{F}{f} \sqrt{a}$ (convex and) non-increasing: but this last condition is generally not satisfied. 
The roughness of this result is due the fact that in this discussion we are not exploiting the quantitative information which relates the transformation $g'$ to $a$.
\end{remark}

\end{document}